\documentclass[11pt]{amsart}
  \usepackage{amsmath, amsfonts, amssymb,amsthm}

\newtheorem{theorem}{Theorem}
\newtheorem{lemma}{Lemma}
\newtheorem{cor}{Corollary}

\theoremstyle{definition}
\newtheorem{exam}[theorem]{Example}

\def\min{\mathop{\mathrm{min}}}

\newcommand{\C}{\mathbb{C}}

\newcommand{\N}{\mathbb{N}}

\newcommand{\PP}{\mathbb{P}}

\newcommand{\ax}{\rightarrow }

\begin{document}
\title{{On Nevanlinna - Cartan theory for holomorphic curves with Tsuji characteristics }}

\author{Nguyen Van Thin}
\address{Department of Mathematics, Thai Nguyen University of Education, 
Luong Ngoc Quyen  street, Thai Nguyen city, Viet Nam.}
\email{thinmath@gmail.com}


\begin{abstract}
In this paper, we prove some fundamental theorems for holomorphic curves on
 $\overline \Omega(\alpha, \beta),$ $\Omega(\alpha, \beta)$ intersecting a hypersurface, finite set of fixed hyperplanes in general position
 and finite set of fixed hypersurfaces in general position on complex projective variety with  the level of truncation. As applications
 of the second main theorems for an angle, we will discuss the uniqueness problem of holomorphic curves in an angle instead of the whole complex
plane. Detail, we establish a result for uniqueness problem of holomorphic curve by inverse image of a hypersurface. In my knowledge, this is the first 
result for uniqueness problem of holomorphic curve by inverse image of hypersurface on angular domain. When
  $\Omega(\alpha,\beta)=\mathbb C,$ we obtain a uniqueness result for holomorphic curves, it is improvement of some results 
 before \cite{DR, P} in this trend.
\end{abstract}
\baselineskip=16truept 
\maketitle 
\pagestyle{myheadings}
\markboth{}{}

\section{ Introduction and main results}

We denote $\Omega(\alpha, \beta)=\{z: \alpha<argz<\beta\}$ by the angle on complex
 plane, where $0<\beta-\alpha\le 2\pi.$ Then, $\Omega(\alpha, \beta)$ is called an angular domain on complex plane. The
  Nevanlinna second main theorem for an angle was used in \cite{MT, DH, GO, GV, WJ, JH2}, and \cite{JH} to investigate the growth of meromorphic functions with some radially
distributed values. The usage of the second main theorem produces a basic and elementary method in the topic \cite{JH2}. 
In \cite{JH}, in view of the Tsuji second main theorem, we established a five-value uniqueness theorem and four-value 
uniqueness theorem for meromorphic functions in an angle. In 2015, J. Zheng \cite{JH1} established the value distribution of holomorphic curves  on an angular domain from the point of view of potential theory and established the first and second fundamental theorems corresponding to those theorems of Ahlfors-Shimizu, Nevanlinna, and Tsuji on meromorphic functions in an angular domain.
We refer readers to \cite{JH1} for comments on the results of the value distribution of holomorphic curves 
on an angular domain. These results motivate us to consider the case of holomorphic curves on
 $\overline \Omega(\alpha, \beta),$ $\Omega(\alpha, \beta)$ intersecting hypersurfaces. In this paper, we prove the fundamental theorems for holomorphic mappings from  
$\overline \Omega(\alpha, \beta)$, $\Omega(\alpha, \beta)$ to $\mathbb P^n(\mathbb C)$ intersecting a hypersurface, 
 finite set of fixed hyperplanes in general position and finite set of fixed hypersurfaces in general position on complex projective variety with the level of
  truncation and the Nevanlinna functions have the form of Tsuji characteristics.

We denote by $k=\dfrac{\pi}{\beta-\alpha},$ and for any pair of real numbers $\alpha$ and $\beta$ in $[0, 2\pi)$ with 
$0<\beta-\alpha\le 2\pi,$
$$\Xi(\alpha, \beta;r)=\{z=te^{i\theta}: \alpha<\theta<\beta, 1<t\le r(sin(k(\theta-\alpha)))^{1/k}\}.$$
 Let $ f: \Omega(\alpha, \beta) \ax \PP^n(\C)$ be a
 holomorphic curve. Let ${\large \text{f}}=(f_0:\dots:f_n)$ be a reduced representation
  of $f,$ where $f_{0},\dots,f_{n}$ are holomorphic functions and  without common zeros in $\Omega(\alpha, \beta).$  Set 
$||{\large \text{f}}(z)||=\max\{|f_0(z)|, \dots, |f_n(z)|\}.$ Let $D$ be a hypersurface in $\PP^n(\C)$ of degree $d$. Let $Q$ be the homogeneous polynomial of degree $d$ defining $D$. 
Under the assumption that $ Q(\large \text{f}) \not\equiv 0,$  the counting function
 $\mathfrak N_{\alpha \beta, f}(r, D)$   of $f$ with respect to $D$ is defined as
\begin{align*}
 \mathfrak N_{\alpha \beta, f}(r, D) &= k\int_{1}^{r}\dfrac{\mathfrak n_{\alpha\beta, f}(t, D)}{t^{k+1}}dt\\
&=\sum_{1<|a_n|<r(sin(k(\theta-\alpha)))^{1/k}}(\dfrac{sink(\alpha_n-\alpha)}{|a_n|^{k}}-\dfrac{1}{r^k}),
\end{align*}
where the $\mathfrak n_{\alpha \beta, f}(t, D)$ are the number zeros of $Q(\large \text{f})$ in the set $\Xi(\alpha, \beta;r)$ counting with  multiplicity
and $a_n=|a_n|e^{i\alpha_n}$ are zeros of $Q(\large \text{f})$ in the set $\Xi(\alpha, \beta;r).$

The {\it  proximity function} of $f$ on $\Omega(\alpha, \beta) $ with respect to $D$ is defined as following:
$$ \mathfrak m_{\alpha \beta, f}(r,D)=\dfrac{1}{2\pi}\int_{arcsin r^{-k}}^{\pi-arcsin r^{-k}} \log \dfrac 
{\|\large \text{f}(rsin^{k^{-1}}
\varphi e^{i(\alpha+k^{-1}\varphi)})\|^d}{|Q(\large \text{f})(rsin^{k^{-1}}\varphi e^{i(\alpha+k^{-1}\varphi)})|}
 \dfrac{d\varphi}{r^k sin^{2}\varphi}.$$
Now let $\delta$ be a positive integer, the {\it truncated  counting function} of $f$ is defined by
\begin{align*}
\mathfrak N_{\alpha \beta, f}^{\delta}(r,D)&=k\int_{1}^{r}\dfrac{\mathfrak n^{\delta}(t, D)}{t^{k+1}}dt\\
&=\sum_{1<|a_n|<r(sin(k(\theta-\alpha)))^{1/k},\; \min\{\text{ord}_{Q(\large \text{f})}(a_n), \delta\}}
(\dfrac{sink(\alpha_n-\alpha)}{|a_n|^{k}}-\dfrac{1}{r^k}),
\end{align*}
where the $\mathfrak n^{\delta}(t, D)$ are the number zeros of $Q(\large \text{f}),$  any zero of multiplicity greater than $\delta$ 
 of $Q(\large \text{f})$ 
in $\Xi(\alpha, \beta;r) $ is ``truncated" and counted as if it only had multiplicity $\delta.$ 

Let $ f: \overline \Omega(\alpha, \beta) \ax \PP^n(\C)$  be a holomorphic
 map. Let ${\large \text{f}}=(f_0:\dots:f_n)$ be a reduced representation
  of $f,$  where $f_{0},\dots,f_{n}$ are holomorphic functions and  without common zeros in $\overline \Omega(\alpha, \beta)$. 
The counting function
 $C_{\alpha \beta, f}(r, D)$   of $f$ with respect to $D$ is defined as
$$ C_{\alpha \beta, f}(r, D) =2\sum_{1\le \rho_n\le r, \alpha\le \psi_n\le \beta}(\dfrac{1}{{\rho_n}^{k}}
-\dfrac{{\rho_n}^{k}}{r^{2k}})sin\hspace{0.2cm}k(\psi_n-\alpha),$$
where the $\rho_ne^{i\psi_n}$ are the zeros of $Q(\large \text{f})$ in $\overline \Omega(\alpha, \beta)$ counting with  multiplicity. For each
 zero  $\rho_ne^{i\psi_n}$ of $Q(\large \text{f})$ in $\overline \Omega(\alpha, \beta)$ with multiple $m$, then
 term $2(\dfrac{1}{{\rho_n}^{k}}-\dfrac{{\rho_n}^{k}}{r^{2k}})sin\hspace{0.2cm}k(\psi_n-\alpha)$ is counted 
$m$ times in $C_{\alpha \beta, f}(r, D).$

Now let $\delta$ be a positive integer, the {\it truncated  counting function} of $f$ is defined by
$$C^\delta_{\alpha\beta, f}(r,D)= 2\sum_{1\le \rho_n\le r, \alpha\le \psi_n\le \beta,\; \min\{\text{ord}_{Q(\large \text{f})}(\rho_ne^{i\psi_n}), \delta\}}(\dfrac{1}
{{\rho_n}^{k}}-\dfrac{{\rho_n}^{k}}{r^{2k}})sin\hspace{0.2cm}k(\psi_n-\alpha),$$
where any zero of multiplicity greater than $\delta$  of $Q(\large \text{f})$ in $\overline \Omega(\alpha, \beta)$ is 
``truncated" and counted 
as if it only had multiplicity $\delta.$ This means that for each  zero  $\rho_ne^{i\psi_n}$ of $Q(f)$ in
 $\overline \Omega(\alpha, \beta)$ with multiple $m$, the terms $2(\dfrac{1}{{\rho_n}^{k}}
-\dfrac{{\rho_n}^{k}}{r^{2k}})sin\hspace{0.2cm}k(\psi_n-\alpha)$ is counted $\min\{m, \delta\}$
 times in $C_{\alpha \beta, f}(r, D).$

The {\it angular proximity Nevanlinna} of $f$ with respect to $D$ is defined as following:
$$A_{\alpha \beta, f}(r,D)=\dfrac{k}{\pi}\int_{1}^{r}(\dfrac{1}{t^k}-\dfrac{t^k}{r^{2k}}) \log \dfrac {\|\large \text{f}(te^{i\alpha})\|^{d}
\|\large \text{f}(te^{i\beta})\|^d}{|Q(\large \text{f})(te^{i\alpha})Q(\large \text{f})(re^{i\beta})|} \dfrac{dt}{t}$$
and 
$$B_{\alpha \beta, f}(r,D)=\dfrac{2k}{\pi r^{k}}\int_{\alpha}^{\beta} \log \dfrac {\|\large \text{f}(re^{i\varphi})\|^{d}}
{|Q(\large \text{f})(re^{i\varphi})|}
sin (k(\varphi-\alpha)) d\varphi,$$
where $\|f(z)\|=\max\{ |f_0(z)|,\dots ,|f_n(z)| \}.$

Let $V\subset \PP^N(\C)$ be a smooth complex projective variety of dimension $n\ge 1.$ Let $D_1, \dots, D_q$ be hypersurfaces
in $\PP^N(\C),$ where $q>n.$ The hypersurfaces $D_1, \dots, D_q$ are said to be {\it in general position on V}\; if for every subset
$\{i_0, \dots, i_n\} \subset \{1, \dots, q\},$ we have
$$ V\cap \text{Supp}D_{i_{0}} \cap \dots \cap \text{Supp}D_{i_n}= \emptyset,$$
where $\text{Supp}(D)$ means the support of the divisor $D.$ A map $f: \Omega(\alpha, \beta) \to V$ is said to be {\it algebraically nondegenerate}
 if the image of $f$ is not contained in any proper subvarieties of $V.$

In this paper,  a notation $``\|"$ in the inequality is mean that the inequality holds for $r\in(1, \infty)$ outside a set with 
measure finite.

Our main results are
\begin{theorem}\label{th1}
Let $D$ be a hypersurface in $\PP^n(\C)$ and $f: \overline \Omega(\alpha, \beta) \ax \PP^{n}(\C)$ be a holomorphic curve whose image is not contained $D$. Then we have for
 any $1< r <\infty$,
$$dS_{\alpha\beta, f}(r)=A_{\alpha \beta, f}(r,D)+B_{\alpha\beta, f}(r,D)+C_{\alpha\beta, f}(r, D)+O(1).$$
\end{theorem}

\begin{theorem}\label{th3} Let $D$ be a hypersurface in $\PP^n(\C)$ and $f:\Omega(\alpha, \beta) 
\ax \PP^{n}(\C)$ be a holomorphic curve whose image is not contained $D$. Then we have for any $1< r <\infty$,
$$d\mathfrak T_{\alpha \beta, f}(r)=\mathfrak m_{\alpha \beta, f}(r,Q)+\mathfrak N_{\alpha \beta, f}(r,Q)+O(1).$$
\end{theorem}

Taking $d=1,$ we get the following results:
\begin{cor}\label{cor1}
Let $H$ be a hyperplane in $\PP^n(\C)$ and $f: \overline \Omega(\alpha, \beta) \ax \PP^{n}(\C)$ be a holomorphic curve whose image is not contained $H$. Then we have for
 any $1< r <\infty$,
$$S_{\alpha\beta, f}(r)=A_{\alpha \beta, f}(r,H)+B_{\alpha\beta, f}(r,H)+C_{\alpha\beta, f}(r, H)+O(1).$$
\end{cor}

\begin{cor}\label{cor3} Let $H$ be a hyperplane in $\PP^n(\C)$ and $f:\Omega(\alpha, \beta) 
\ax \PP^{n}(\C)$ be a holomorphic curve whose image is not contained $H$. Then we have for any $1< r <\infty$,
$$\mathfrak T_{\alpha \beta, f}(r)=\mathfrak m_{\alpha \beta, f}(r,H)+\mathfrak N_{\alpha \beta, f}(r,H)+O(1).$$
\end{cor}

\begin{theorem}\label{th4}  Let $ f: \C \ax \PP^{n}(\C)$  be a linearly non-degenerate holomorphic
 curve  and $H_{1},\dots , H_{q}$ be hyperplanes in $\PP^{n}(\C)$ in general position. Then we have
$$\| \quad (q-n-1)S_{\alpha\beta, f}(r)\leqslant\sum_{j=1}^{q} C^{n}_{\alpha\beta, f}(r,H_j)+O(\log T_f(r)+\log r).$$
\end{theorem}

\begin{theorem}\label{th6} Let $ f: \Omega(\alpha, \beta) \ax \PP^{n}(\C)$  be a linearly 
non-degenerate holomorphic curve  and $H_{1},\dots , H_{q}$ be hyperplanes in $\PP^{n}(\C)$ in general position. Then we have
$$\| \quad (q-n-1)\mathfrak T_{ \alpha \beta, f}(r)\leqslant\sum_{j=1}^{q} \mathfrak {N^{n}}_{\alpha \beta, f}(r,H_j)
+O(\log \mathfrak T_{\alpha \beta, f}(r)+\log r).$$
\end{theorem}

\begin{theorem}\label{Th10} Let $f: \Omega(\alpha, \beta) \to \mathbb P^N(\mathbb C)$ be an algebraically nondegenerate
holomorphic curve. Let $d $ and $n$ be two integers with $n>N(d+N+1).$ Let 
$\mathcal H_i=\{z\in \mathbb P^{N}(\mathbb C), \mathcal H_i(z)=0\}, 0\le i\le N,$ be hyperplanes in $\mathbb P^N(\mathbb C).$ Let $D_i=\{z\in \mathbb P^{N}(\mathbb C), Q_i(z)=0\}, 
0\le i\le N,$ be hypersurfaces of degree $d$ such that the hypersurfaces
 $\{\mathcal H_0^nQ_0=0\}, \dots, \{\mathcal H_N^nQ_N=0\}$ are in general position in 
$\mathbb P^N(\mathbb C).$ Let 
 $D=\{z\in \mathbb P^N(\mathbb C), \sum_{i=0}^{N}\mathcal H_i^nQ_i=0\}.$ Then
\begin{align*}
\|(n-(d+N+1)N)\mathfrak T_{\alpha\beta, f}(r)&+\sum_{i=0}^{N}(\mathfrak N_{\alpha\beta, f}(r, D_i)
-\mathfrak N_{\alpha\beta, f}^{N}(r, D_i))\\
&\le \mathfrak N_{\alpha\beta, f}^{N}(r, D)+o(\mathfrak T_{\alpha\beta, f}(r)).
\end{align*}
\end{theorem}

We give a hypersurfaces satisfying Theorem \ref{Th10}.
\begin{exam}\label{exam2} Let $D_i=\{z=(x_0:\dots:x_N)\in \mathbb P^{N}(\mathbb C),  x_i^d=0\}, 
0\le i\le N,$ be hypersurfaces of degree $d.$  Let $\mathcal H_i=\{z=(x_0:\dots:x_N)\in \mathbb P^{N}(\mathbb C), 
\sum_{t=0}^{i}x_t=0\}.$
We see that the hypersurfaces
 $\{(\sum_{t=0}^{i}x_t)^{n}x_i^{d}=0\}, 0\le i\le N,$ are in general position in $\mathbb P^N(\mathbb C).$ Then 
 $$D=\{z\in \mathbb P^N(\mathbb C), \sum_{i=0}^{N}(\sum_{t=0}^{i}x_t)^{n}x_i^{d}=0\}$$ satisfies the 
Theorem \ref{Th10} with $n>N(d+N+1).$
\end{exam}

As an application of Theorem \ref{Th10}, we prove the uniqueness theorem for holomorphic curves on angular domain by
 inverse image of a hypersurface.
\begin{theorem}\label{Th11}
Let $f, g :  \Omega(\alpha, \beta) \to \mathbb P^N(\mathbb C)$ be two algebraically nondegenerate
holomorphic curves, and $n$ be a integer with $n>N(d+N+3).$ Let $D$ be a hypersurface as the same Theorem \ref{Th10}.  
Suppose that $f(z)=g(z)$ on $f^{-1}(D)\cup g^{-1}(D),$ then $f\equiv g.$
\end{theorem}

In my knowledge, up to now, Theorem \ref{Th11} is a first result for uniqueness problem of holomorphic curve by 
inverse image of a hypersurface on angular domain. 

When $\alpha=0, \beta=2\pi,$ this means $\Omega(\alpha, \beta)=\mathbb C,$ we obtain some uniqueness results for holomorphic
 curves on complex plane as  following:

\begin{cor}\label{CorTh11}
Let $f, g:  \mathbb C \to \mathbb P^N(\mathbb C)$ be two algebraically nondegenerate
holomorphic curves, and $n$ be a integer with $n>N(d+N+3).$ Let $D$ be a hypersurface as the same Theorem \ref{Th10}.  
Suppose that $f(z)=g(z)$ on $f^{-1}(D)\cup g^{-1}(D),$ then $f\equiv g.$
\end{cor}

By using method of Ru \cite{Ru3} and Ru et. al. \cite{Ru4}, we are easy to get some results as follows:

\begin{theorem}\label{th7} Let $V\subset \PP^{N}(\C)$ be a complex projective variety of dimension $n\ge 1.$ Let 
$D_1, \dots, D_q$ be hypersurfaces in $\PP^N(\C)$ of degree $d_j,$ located in general position on $V.$ Let $d$ be the least 
common
 multiple of the $d_i,$ $i=1, \dots, q.$ Let $ f: \Omega(\alpha, \beta)  \ax V$  be an algebraically non-degenerate
 holomorphic map. Let $\varepsilon >0$
 and $$ M \ge \dfrac{n^nd^{n^2+n}(19nI(\varepsilon^{-1}))^n(\deg V)^{n+1}}{n!},$$
where $I(x):=\min\{k\in \N: k>x\}$ for a positive real number $x.$ Then
$$\quad (q(1-\varepsilon/3)-(n+1)-\varepsilon/3)\mathfrak T_{ \alpha \beta, f}(r)\leqslant\sum_{l=1}^{q}d_l^{-1} 
{\mathfrak {N}^{M}}_{\alpha \beta, f}(r,Q_l)+O(\log \mathfrak T_{\alpha \beta, f}(r)+\log r)$$
holds for all $r\in (0, +\infty)$ outside a set of finite measure.
\end{theorem}

\begin{theorem}\label{th9} Let $V\subset \PP^{N}(\C)$ be a complex projective variety of dimension $n\ge 1.$ Let 
$D_1, \dots, D_q$ be hypersurfaces in $\PP^N(\C)$ of degree $d_j,$ located in general position on $V.$ Let $d$ be the least 
common
 multiple of the $d_i,$ $i=1, \dots, q.$ Let $ f: \mathbb C  \ax V$  be an algebraically non-degenerate
 holomorphic map. Let $\varepsilon >0$
 and $$ M \ge \dfrac{n^nd^{n^2+n}(19nI(\varepsilon^{-1}))^n(\deg V)^{n+1}}{n!},$$
where $I(x):=\min\{k\in \N: k>x\}$ for a positive real number $x.$ Then
$$\quad (q(1-\varepsilon/3)-(n+1)-\varepsilon/3)S_{\alpha\beta, f}(r)\leqslant\sum_{l=1}^{q}d_l^{-1} 
C^{M}_{\alpha \beta, f}(r,Q_l)+O(\log  T_{f}(r)+\log r)$$
holds for all $r\in (0, +\infty)$ outside a set of finite measure.
\end{theorem}

\section{Some preliminaries in angular Nevanlinna theory for meromorphic functions}
\def\theequation{2.\arabic{equation}}
\setcounter{equation}{0} 
First, we remind some definitions which is contained the book of  A. A. Goldberg and I. V. Ostrovskii. We consider the set 
$$\Omega(\alpha, \beta; r)=\Omega(\alpha, \beta)\cap\{1<|z|<r\}.$$
 Let $f$ be a meromorphic function on the angle $\overline\Omega (\alpha, \beta;r)$, $0<\beta-\alpha\le 2\pi$, $1\le r<\infty.$
 We recall that

\begin{align*}
A_{\alpha \beta}(r,f)&=\dfrac{k}{\pi} \int_1^{r}(\dfrac{1}{t^k}-\dfrac{t^k}{r^{2k}})[\log^{+}|f(te^{i\alpha})|
+\log^{+}|f(te^{i\beta})|]\dfrac{dt}{t};\\
 B_{\alpha\beta}(r, f)&=\dfrac{2k}{\pi r^k}\int_{\alpha}^{\beta}\log^{+}|f(re^{i\varphi})|.sin(k(\varphi-\alpha))d\varphi;\\
C_{\alpha\beta}(r, f)&=2k\int_{1}^{r}c_{\alpha\beta}(r, f)(\dfrac{1}{t^k}+\dfrac{t^k}{r^{2k}})\dfrac{dt}{t}\\
&=2\sum_{1\le \rho_n\le r, \alpha\le \psi_n\le \beta}(\dfrac{1}{{\rho_n}^{k}}-\dfrac{{\rho_n}^{k}}{r^{2k}})
sin\hspace{0.2cm}k(\psi_n-\alpha),
\end{align*}
where
 $$ c_{\alpha \beta}(r, f) =\sum_{1<\rho_n\le r, \alpha \le \psi_n\le \beta}sin(k(\psi_n-\alpha)),$$
and $\rho_ne^{i\psi_n}$ are poles of $f(z)$ counted according with multiplicity.
We denote $S_{\alpha \beta}(r, f)$ by the angular Nevanlinna characteristics on $\overline \Omega(\alpha, \beta;r)$ and
 defined as following:
$$ S_{\alpha \beta}(r, f) =A_{\alpha \beta}(r, f)+B_{\alpha \beta}(r, f)+C_{\alpha \beta}(r, f).$$
In order to prove  theorems, we need the following lemmas.
\begin{lemma}\cite{GO}{\bf{(Carleman formula)}}\label{lm1}  Let $f$ be a nonconstant meromorphic function in
 $\overline \Omega(\alpha, \beta;r).$ Then
\begin{align*}
C_{\alpha \beta}(r, \dfrac{1}{f}) -C_{\alpha \beta}(r, f)&= \dfrac{k}{\pi} \int_1^{r}(\dfrac{1}{t^k}-\dfrac{t^k}{r^{2k}})
[\log |f(te^{i\alpha})|+\log |f(te^{i\beta})|]\dfrac{dt}{t}\\
&+\dfrac{2k}{\pi r^k}\int_{\alpha}^{\beta}\log |f(re^{i\varphi})|.sin(k(\varphi-\alpha))d\varphi+O(1).
\end{align*}
\end{lemma}

For any pair of real numbers $\alpha$ and $\beta$ in $[0, 2\pi)$ with $0<\beta-\alpha\le 2\pi,$
$$\Xi(\alpha, \beta;r)=\{z=te^{i\theta}: \alpha<\theta<\beta, 1<t\le r(sin(k(\theta-\alpha)))^{1/k}\}.$$

Let $f$ be a nonconstant meromorphic function on $\Omega(\alpha, \beta)$. We define 
\begin{align*}
 \mathfrak N_{\alpha \beta}(r, f) &= k\int_{1}^{r}\dfrac{\mathfrak n_{\alpha\beta}(t, f)}{t^{k+1}}dt\\
&=\sum_{1<|b_n|<r(sin(k(\theta-\alpha)))^{1/k}}(\dfrac{sin k(\beta_n-\alpha)}{|b_n|^{k}}-\dfrac{1}{r^k}),
\end{align*}
where the $\mathfrak n_{\alpha \beta}(t, f)$ are the number poles of $f$ in the set $\Xi(\alpha, \beta;t)$ counting with  multiplicity
and $b_n=|b_n|e^{i\beta_n}$ are poles of $f$ in the set $\Xi(\alpha, \beta;t).$ 

The {\it  proximity function} of $f$ on $\Omega(\alpha, \beta) $  is given by
$$ \mathfrak m_{\alpha \beta}(r,f)=\dfrac{1}{2\pi}\int_{arcsin r^{-k}}^{\pi-arcsin r^{-k}} \log^{+}
|f(rsin^{k^{-1}}
\varphi e^{i(\alpha+k^{-1}\varphi)})|\dfrac{d\varphi}{r^k sin^{2}\varphi}.$$
The {\it characteristic function} $\mathfrak T_{\alpha \beta}(r, f)$ of $f$ on $\Omega(\alpha, \beta)$ is defined by 
$$\mathfrak T_{\alpha\beta}(r,f)=\mathfrak m_{\alpha \beta}(r, f)+\mathfrak N_{\alpha \beta}(r, f).$$

\begin{lemma}\cite{JH}\label{lm3}  Let $f$ be a nonconstant meromorphic function in $\Omega(\alpha, \beta).$ Then
\begin{align*}
\mathfrak N_{\alpha \beta}(r, \dfrac{1}{f})- \mathfrak N_{\alpha \beta}(r, f)=\dfrac{1}{2\pi}
\int_{arcsin r^{-k}}^{\pi-arcsin r^{-k}} \log |f(rsin^{k^{-1}}\varphi e^{i(\alpha+k^{-1}\varphi)})|\dfrac{d\varphi}
{r^k sin^{2}\varphi}+O(1).
\end{align*}
\end{lemma}

\begin{lemma}\cite{GO}\label{lm4}  Let $k$ be a natural number and $f$ be nonconstant meromorphic function on $\C.$ Then
we have the estimate
$$S_{\alpha \beta}(r, \dfrac{f^{(k)}}{f})\le O(\log T(r, f)+\log r)$$
 holds for $1<r<\infty$ outside a set of finite measure.
\end{lemma}

\begin{lemma}\cite{JH}\label{lm6}   Let $k$ be a natural number and $f$ be nonconstant meromorphic function on $\Omega(\alpha, \beta).$ Then we have the estimate
$$\mathfrak m_{\alpha \beta}(r, \dfrac{f^{(k)}}{f})\le O(\log \mathfrak T_{\alpha \beta}(r, f)+\log r)$$
 holds for $1<r<\infty$ outside a set of finite measure.
\end{lemma}

\section{Proofs of Theorems}
\def\theequation{3.\arabic{equation}}
\setcounter{equation}{0}

\begin{proof}[Proof of Theorem \ref{th1} and Theorem \ref{th3}]
First, we prove the Theorem \ref{th1}. Note that $C_{\alpha\beta}(r, Q(f))=0.$ By the definitions of 
$S_{\alpha\beta, f}(r)$, $A_{\alpha \beta, f}(r,D)$, $ B_{\alpha\beta, f}(r,D)$
 and apply to Lemma \ref{lm1} for $Q(f) \not\equiv 0$, we have
\begin{align*}
C_{\alpha \beta, f}(r,D) &= \dfrac{k}{\pi} \int_1^{r}(\dfrac{1}{t^k}-\dfrac{t^k}{r^{2k}})
[\log |Q(f)(te^{i\alpha})|+\log |Q(f)(te^{i\beta})|]\dfrac{dt}{t}\\
&+\dfrac{2k}{\pi r^k}\int_{\alpha}^{\beta}\log |Q(f)(re^{i\varphi})|.sin(k(\varphi-\alpha))d\varphi+O(1).
\end{align*}
Hence, we get
\begin{align*}
A_{\alpha \beta, f}(r,D)&+B_{\alpha \beta, f}(r,D)+C_{\alpha \beta, f}(r, D)\\
&=\dfrac{k}{\pi}\int_{1}^{r}(\dfrac{1}{t^k}-\dfrac{t^k}{r^{2k}}) \log \dfrac {\|f(te^{i\alpha})\|^d\|f(te^{i\beta})\|^d}
{|Q(f)(te^{i\alpha})Q(f)(re^{i\beta})|} \dfrac{dt}{t}\\
&\quad +\dfrac{2k}{\pi r^{k}}\int_{\alpha}^{\beta} \log \dfrac {\|f(re^{i\varphi})\|^d}{|Q(f)(re^{i\varphi})|}
sin (k(\varphi-\alpha)) d\varphi\\
&\quad +\dfrac{k}{\pi} \int_1^{r}(\dfrac{1}{t^k}-\dfrac{t^k}{r^{2k}})[\log |Q(f)(te^{i\alpha})|+\log |Q(f)(te^{i\beta})|]
\dfrac{dt}{t}\\
&\quad+\dfrac{2k}{\pi r^k}\int_{\alpha}^{\beta}\log |Q(f)(re^{i\varphi})|.sin(k(\varphi-\alpha))d\varphi+O(1)
\end{align*}
\begin{align*}
&=\dfrac{k}{\pi}\int_{1}^{r}(\dfrac{1}{t^k}-\dfrac{t^k}{r^{2k}}) \log \|f(te^{i\alpha})\|^d\|f(te^{i\beta})\|^d \dfrac{dt}{t}\\
&\quad+\dfrac{2k}{\pi r^{k}}\int_{\alpha}^{\beta} \log \|f(re^{i\varphi})\|^d. sin (k(\varphi-\alpha)) d\varphi+O(1)\\
&=dS_{\alpha \beta, f}(r)+O(1).
\end{align*}
This is conclusion of Theorem \ref{lm1}.

The end, we prove Theorem \ref{th3}. We have $\mathfrak N_{\alpha\beta}(r, Q(f))=0.$ By the definitions of 
$\mathfrak T_{\alpha\beta, f}(r)$, $\mathfrak m_{\alpha \beta, f}(r,D)$, $ \mathfrak N_{\alpha\beta, f}(r,D)$
 and apply to Lemma \ref{lm3} for $Q(f) \not\equiv 0$, we have
\begin{align*}
\mathfrak N_{\alpha \beta, f}(r,D) &=\dfrac{1}{2\pi}\int_{arcsin r^{-k}}^{\pi-arcsin r^{-k}}
 \log |Q(f)(rsin^{k^{-1}}\varphi e^{i(\alpha+k^{-1}\varphi)})|\dfrac{d\varphi}{r^k sin^{2}\varphi}+O(1).
\end{align*}
Thus, we obtain
\begin{align*}
\mathfrak m_{\alpha \beta, f}(r,D)&+\mathfrak N_{\alpha \beta, f}(r,D)\\
&=\dfrac{1}{2\pi}\int_{arcsin r^{-k}}^{\pi-arcsin r^{-k}} \log |Q(f)(rsin^{k^{-1}}\varphi e^{i(\alpha+k^{-1}\varphi)})| \dfrac{d\varphi}{r^k sin^2\varphi}\\
&\quad +\dfrac{1}{2\pi}\int_{arcsin r^{-k}}^{\pi-arcsin r^{-k}} \log \dfrac {
\|f(rsin^{k^{-1}}\varphi e^{i(\alpha+k^{-1}\varphi)})\|^d}{|Q(f)(
rsin^{k^{-1}}\varphi e^{i(\alpha+k^{-1}\varphi)})|} \dfrac{d\varphi}{r^k sin^{2}\varphi}+O(1)\\
&=\dfrac{1}{2\pi}\int_{arcsin r^{-k}}^{\pi-arcsin r^{-k}} \log\|f(rsin^{k^{-1}}\varphi e^{i(\alpha+k^{-1}\varphi)})\|^d \dfrac{d\varphi}{r^k sin^2\varphi}+O(1)\\
&=d\mathfrak T_{\alpha \beta, f}(r)+O(1).
\end{align*}
We have completed the proof of Theorem \ref{th3}.

\end{proof}

In order to prove the Theorem \ref{th4}, and Theorem \ref{th6}, we need some lemmas. 
First we recall the property of Wronskian.

Let  $f_0, f_1, \dots, f_n$ be meromorphic functions on complex plane $\mathbb C,$
then Wronskian of $f_0, f_1, \dots, f_n$ is defined by
\[  W(f_0,\dots,f_n) =
\left|\begin{array}{cccc}
f_{0}(z)&f_{1}(z)&\cdot & f_{n}(z)\\
f'_{0}(z)&f'_{1}(z)&\cdot & f'_{n}(z)\\
\vdots&\vdots&\ddots&\vdots\\
f^{(n)}_{0}(z)&f^{(n)}_{1}(z)&\cdot & f^{(n)}_{n}(z)\\
\end{array}
\right|.
\]
\begin{lemma}\cite{IL} \label{lm10} 
Let $f_0, f_1, \dots, f_n$ be meromorphic functions on $\mathbb C,$ then
$$ W(f_0, f_1, \dots, f_n)=f_0^{n+1}W(1, \dfrac{f_1}{f_0}, \dots, \dfrac{f_n}{f_0}).$$
\end{lemma}
 Let $f=(f_0:\dots:f_n): \mathbb C \ax \PP^n(\mathbb C)$ be holomorophic curve, then Wronskian of $f$ is defined by
$$W = W(f) = W(f_0,\dots,f_n).$$

We denote $C_{\alpha \beta, W}(r,0)$ by the counting function in zeros of $W(f_0,\dots,f_n)$ in $\overline \Omega(\alpha, \beta),$
 this means
$$C_{\alpha \beta, W}(r,0) = C_{\alpha\beta}(r,\dfrac{1}{W})+O(1).$$
We use the notation $\mathfrak N_{W}(r,0)$ talking the counting function in zeros of $W(f_0,\dots,f_n)$ in $\text{Imz}\ge 0,$
 namely
$$ \mathfrak N_{W}(r,0) = \mathfrak N(r,\dfrac{1}{W})+O(1).$$
We call $\mathfrak N_{\alpha \beta, W}(r,0)$ the counting function in zeros of $W(f_0,\dots,f_n)$ in $ \Omega(\alpha, \beta),$ namely
$$ \mathfrak N_{\alpha \beta, W}(r,0) = \mathfrak N_{\alpha\beta}(r,\dfrac{1}{W})+O(1).$$

Let $L_0,\dots,L_{n}$ are linearly independent forms  of $z_0,\dots,z_n$. For $j=0,\dots,n$, set $$F_j(z):= L_j(f(z)).$$ By the property of Wronskian there exists a constant $C \ne 0$ such that
$$|W(F_0,\dots,F_n)| =C|W(f_{0},\dots ,f_{n})|.$$

\begin{lemma}\label{lm11} Let $ f=(f_{0}:\dots:f_{n}) : \C \ax \PP^{n}(\C)$  be a linearly non-degenerate
 holomorphic curve  and $H_{1},\dots , H_{q}$ be hyperplanes in $\PP^{n}(\C)$ in general position. Then we have
\begin{align*}
\| \quad &\dfrac{k}{\pi}\int_{1}^{r}(\dfrac{1}{t^k}-\dfrac{t^k}{r^{2k}})\max_{K}\sum_{j \in K} \log \dfrac 
{\|f(te^{i\alpha})\|}{|(a_j,f)(te^{i\alpha})|} \dfrac{dt}{t} \\
&\quad+ \dfrac{k}{\pi}\int_{1}^{r}(\dfrac{1}{t^k}-\dfrac{t^k}{r^{2k}})\max_{K}\sum_{j \in K} \log \dfrac 
{\|f(te^{i\beta})\|}{|(a_j,f)(te^{i\beta})|} \dfrac{dt}{t} \\
&\quad+\dfrac{2k}{\pi r^{k}}\int_{\alpha}^{\beta}\max_{K}\sum_{j \in K} \log \dfrac {\|f(re^{i\varphi})\|}{|(a_j,f)
(re^{i\varphi})|}sin (k(\varphi-\alpha)) d\varphi\\
&\leqslant (n+1)S_{\alpha\beta, f}(r)- C_{\alpha\beta, W}(r,0) +O(\log T_f(r)+\log r).
\end{align*}
Here the maximum is taken over all subsets $K$ of $\{1,\dots,q\}$ such that $a_{j}$, $j\in K$, are linearly independent.
\end{lemma}

\begin{proof}
First, we prove
\begin{align}\label{ct3a1}
\int_{\alpha}^{\beta}&\max_{K}\sum_{j \in K} \log \dfrac {\|f(re^{i\varphi})\|}{|(a_j,f)
(re^{i\varphi})|}sin (k(\varphi-\alpha)) d\varphi\notag\\
&\le (n+1) \int_{\alpha}^{\beta} \log \|f(re^{i\varphi})\|. sin (k(\varphi-\alpha)) 
d\varphi\notag\\
&\quad-(n+1)\int_{\alpha}^{\beta} \log |W(f_0, \dots, f_n)(re^{i\varphi})|. sin (k(\varphi-\alpha)) d\varphi\notag\\
&\quad +O(\log T_f(r)+\log r).
\end{align}
Let $K \subset \{1,\dots,q\}$ such that $a_j, j \in K,$ are linearly independent. Without loss 
of generality, we may assume that $ q\geqslant n+1$ and $\# K=n+1.$ Let $\mathcal T$ is the set of all injective maps
 $\mu:\{ 0,1,\dots,n\} \ax \{1,\dots,q\}$. Then we have
\begin{align*}
\int_{\alpha}^{\beta}\max_{K}\sum_{j \in K}&\log \dfrac {\|f(re^{i\varphi})\|}{|(a_j,f)(re^{i\varphi})|}. sin (k(\varphi-\alpha)) d\varphi \\
&= \int_{\alpha}^{\beta}\max_{\mu \in \mathcal  T}\sum_{j=0}^n\log \dfrac {\|f(re^{i\varphi})\|}{|(a_{\mu(j)},f)(re^{i\varphi})|}.
sin (k(\varphi-\alpha)) d\varphi\\
&= \int_{\alpha}^{\beta} \log \bigg\{\max_{\mu \in \mathcal  T} \dfrac {\|f(re^{i\varphi})\|^{n+1}}{\prod\limits_{j=0}^n
|(a_{\mu(j)},f)(re^{i\varphi})|}\bigg\} sin (k(\varphi-\alpha)) d\varphi +O(1)\\
&\leqslant \int_{\alpha}^{\beta} \log \sum_{\mu \in  \mathcal  T}\dfrac {\|f(re^{i\varphi})\|^{n+1}}{\prod\limits_{j=0}^n
|(a_{\mu(j)},f)
(re^{i\varphi})|} sin (k(\varphi-\alpha)) d\varphi+O(1).\\
\end{align*}
Thus, we obtain
\begin{align*}
\int_{\alpha}^{\beta}&\max_{K}\sum_{j \in K} \log \dfrac {\|f(re^{i\varphi})\|}{|(a_j,f)(re^{i\varphi})|}sin (k(\varphi-\alpha)) d\varphi 
 \\
&\le \int_{\alpha}^{\beta} \log \sum_{\mu \in \mathcal  T}\dfrac {|W((a_{\mu(0)},f),\dots,(a_{\mu(n)},f))(re^{i\varphi})|}
{\prod\limits_{j=0}^n|(a_{\mu(j)},f)(re^{i\varphi})|}sin (k(\varphi-\alpha)) d\varphi\\
& \quad+\int_{\alpha}^{\beta}\log \sum_{\mu \in \mathcal  T}\dfrac{\|f(re^{i\varphi})\|^{n+1}}
{|W((a_{\mu(0)},f),\dots,(a_{\mu(n)},f))(re^{i\varphi})|} sin (k(\varphi-\alpha)) d\varphi+O(1).
\end{align*}
By the property of Wronskian, we see that
$$|W((a_{\mu(0)},f),\dots,(a_{\mu(n)},f))| =C|W(f_{0},\dots ,f_{n})|,$$
where $C \ne 0$ is constant. 

%So we obtain
%\begin{align*}
%\int_{\alpha}^{\beta}&\max_{K}\sum_{j \in K}\log \dfrac {\|f(re^{i\varphi})\|}{|(a_j,f)(re^{i\varphi})|} sin (k(\varphi-\alpha)) d\varphi \\
%&\leqslant \int_{\alpha}^{\beta}\log \sum_{\mu \in T}\dfrac {|W((a_{\mu(0)},f),\dots,(a_{\mu(n)},f))(re^{i\varphi})|}
%{\prod\limits_{j=0}^n|(a_{\mu(j)},f)(re^{i\varphi})|} sin (k(\varphi-\alpha)) d\varphi\\
%& \qquad+\int_{\alpha}^{\beta}\log \dfrac{\|f(re^{i\varphi})\|^{n+1}} {|W(f_0,\dots,f_n)(re^{i\varphi})|} sin (k(\varphi-\alpha)) d\varphi
%+O(1).
%\end{align*}
Thus, we get
\begin{align}\label{ct3.2}
\int_{\alpha}^{\beta}&\max_{K}\sum_{j \in K}\log \dfrac {\|f(re^{i\varphi})\|}{|(a_j,f)(re^{i\varphi})|} sin (k(\varphi-\alpha)) d\varphi\notag \\
&\leqslant \int_{\alpha}^{\beta}\log \sum_{\mu \in T}\dfrac {|W((a_{\mu(0)},f),\dots,(a_{\mu(n)},f))(re^{i\varphi})|}
{\prod\limits_{j=0}^n|(a_{\mu(j)},f)(re^{i\varphi})|} sin (k(\varphi-\alpha)) d\varphi\notag \\
& \qquad+\int_{\alpha}^{\beta}\log \dfrac{\|f(re^{i\varphi})\|^{n+1}} {|W(f_0,\dots,f_n)(re^{i\varphi})|}
 sin (k(\varphi-\alpha)) d\varphi+O(1).
\end{align}

Take
$$ g_{\mu(j)}= \dfrac{(a_{\mu(j)},f)}{(a_{\mu(0)},f)}, j=1, \dots, n.$$
Apply to Lemma \ref{lm10},  we have
\begin{align}\label{ct3a2}
\dfrac {W((a_{\mu(0)},f),\dots,(a_{\mu(n)},f))}{\prod\limits_{j=0}^n(a_{\mu(j)},f)} 
&= \dfrac{W(1, \dfrac{(a_{\mu(1)},f)}{(a_{\mu(0)},f)}, \dots, \dfrac{(a_{\mu(n)},f)}{(a_{\mu(0)},f)})}
{\dfrac{(a_{\mu(1)},f)}{(a_{\mu(0)},f)}\dots \dfrac{(a_{\mu(n)},f)}{(a_{\mu(0)},f)} }\notag\\
&=\left|\begin{array}{cccc}
1& 1&\dots & 1\\
0& \dfrac{g'_{\mu(1)}}{ g_{\mu(1)}}&\dots & \dfrac{g'_{\mu(n)}}{ g_{\mu(n)}}\\
\vdots&\vdots&\ddots&\vdots\\
0 & \dfrac{g^{(n)}_{\mu(1)}}{ g_{\mu(1)}}&\dots & \dfrac{g^{(n)}_{\mu(n)}}{ g_{\mu(n)}}
\end{array}
\right|.
\end{align}
For each $j\in \{1, \dots, n\}$ and $k\in \mathbb N^{*},$ using Lemma \ref{lm4}, we have the inequality as following

\begin{align}\label{ct3.3aa}
\| \quad B_{\alpha\beta}(r, \dfrac{g_{\mu(j)}^{(k)}}{g_{\mu(j)}})\le S_{\alpha\beta}\bigg(r,
\dfrac{g_{\mu(j)}^{(k)}}{g_{\mu(j)}}\bigg)&\leqslant O(\log r + \log T(r, g_{\mu(j)})).
\end{align}
Futhermore $T(r, g_{\mu(j)})\le T_f(r)+O(1).$ Then from (\ref{ct3.3aa}), we have
$$\| \quad B_{\alpha\beta}(r, \dfrac{g_{\mu(j)}^{(k)}}{g_{\mu(j)}})
 \leqslant O(\log r+\log T_f(r)).$$
Hence for any $\mu \in \mathcal  T$ and from (\ref{ct3a2}), we have 
\begin{align*}
\| \quad \int_{\alpha}^{\beta}\log^+&\dfrac {|W((a_{\mu(0)},f),\dots,(a_{\mu(n)},f))(re^{i\varphi})|}
{\prod\limits_{j=0}^n|(a_{\mu(j)},f)(re^{i\varphi})|} sin (k(\varphi-\alpha)) d\varphi\\
 &\leqslant O(\log r+\log T_{f}(r)).
\end{align*}
This implies that
\begin{align}\label{ct3.4}
\| \quad \int_{\alpha}^{\beta} \log &\sum_{\mu \in \mathcal  T} 
\dfrac {|W((a_{\mu(0)},f),\dots,(a_{\mu(n)},f))(re^{i\varphi})|}
{\prod\limits_{j=0}^n|(a_{\mu(j)},f)(re^{i\varphi})|}sin (k(\varphi-\alpha)) d\varphi \\
&\leqslant \int_{\alpha}^{\beta} \log^+\sum_{\mu \in \mathcal  T}
 \dfrac {|W((a_{\mu(0)},f),\dots,(a_{\mu(n)},f))(re^{i\varphi})|}
{\prod\limits_{j=0}^n|(a_{\mu(j)},f)(re^{i\varphi})|}sin (k(\varphi-\alpha)) d\varphi \notag\\
&\leqslant  \sum_{\mu \in \mathcal  T}\int_{\alpha}^{\beta}  \log^+\dfrac
 {|W((a_{\mu(0)},f),\dots,(a_{\mu(n)},f))(re^{i\varphi})|}
{\prod\limits_{j=0}^n|(a_{\mu(j)},f)(re^{i\varphi})|}sin (k(\varphi-\alpha)) d\varphi+O(1) \notag\\
& \leqslant O(\log r+\log T_{f}(r)).\notag
\end{align}
Combining (\ref{ct3.2}) and (\ref{ct3.4}), we get the inequality (\ref{ct3a1}). Similarly, we obtain
\begin{align}\label{ct3a3}
\int_{1}^{r}&\max_{K}\sum_{j \in K} (\dfrac{1}{t^k}-\dfrac{t^k}{r^{2k}})\log \dfrac {\|f(te^{i\alpha})\|}{|(a_j,f)
(te^{i\alpha})|} \dfrac{dt}{t}\notag\\
&\le (n+1) \int_{1}^{r} (\dfrac{1}{t^k}-\dfrac{t^k}{r^{2k}}) \log \|f(te^{i\alpha})\| \dfrac{dt}{t} \notag\\
&\quad-(n+1)\int_{1}^{r} (\dfrac{1}{t^k}-\dfrac{t^k}{r^{2k}}) \log |W(f_0, \dots, f_n)(te^{i\alpha})|. \dfrac{dt}{t}\notag\\
&\quad +O(\log T_f(r)+\log r).
\end{align}
and
\begin{align}\label{ct3a4}
&\int_{1}^{r}\max_{K}\sum_{j \in K} (\dfrac{1}{t^k}-\dfrac{t^k}{r^{2k}})\log \dfrac {\|f(te^{i\beta})\|}{|(a_j,f)
(te^{i\beta})|} \dfrac{dt}{t}\notag\\
&\le (n+1) \int_{1}^{r} (\dfrac{1}{t^k}-\dfrac{t^k}{r^{2k}}) \log \|f(te^{i\beta})\| \dfrac{dt}{t} \notag\\
&\quad -(n+1)\int_{1}^{r} (\dfrac{1}{t^k}-\dfrac{t^k}{r^{2k}}) \log |W(f_0, \dots, f_n)(te^{i\beta})|. \dfrac{dt}{t}\notag\\
 &\quad +O(\log T_f(r)+\log r).
\end{align}
We may obtain the conclusion of Lemma \ref{lm11} by adding (\ref{ct3a1}), (\ref{ct3a3}) and (\ref{ct3a4}) and note that 
\begin{align*}
C_{\alpha\beta, W}(r,0)&=\dfrac{2k}{\pi r^k}\int_{\alpha}^{\beta} \log |W(f_0, \dots, f_n)(re^{i\varphi})|. sin (k(\varphi-\alpha))
 d\varphi\\
&\quad + \dfrac{k}{\pi}\int_{1}^{r} (\dfrac{1}{t^k}-\dfrac{t^k}{r^{2k}}) \log |W(f_0, \dots, f_n)(te^{i\alpha})|. \dfrac{dt}{t}\\
&\quad +\dfrac{k}{\pi} \int_{1}^{r} (\dfrac{1}{t^k}-\dfrac{t^k}{r^{2k}}) \log |W(f_0, \dots, f_n)(te^{i\beta})|.
 \dfrac{dt}{t}.
\end{align*}
We have completed the proof of this lemma.
\end{proof}

\begin{lemma}\label{lm12} Let $ f=(f_{0}:\dots:f_{n}) : \C \longrightarrow \PP^{n}(\C)$  be a linearly
 non-degenerate holomorphic curve  and $H_{1},\dots ,H_{q}$ be hyperplanes in $\PP^{n}(\C)$ in general position. 
Let $a_j$ be the vector associated with $H_j$ for $j=1,\dots,q$. Then
\begin{align*}
\sum_{j=1}^{q}[A_{\alpha\beta, f}(r,H_{j})&+B_{\alpha\beta, f}(r, H_j)] \\
&\leqslant \dfrac{2k}{\pi r^{k}}
\int_{\alpha}^{\beta}(\dfrac{1}{t^k}-\dfrac{t^k}{r^{2k}})\max_{K}\sum_{j \in K} \log \dfrac {\|f(re^{i\varphi})\|}{|(a_j,f)
(re^{i\varphi})|}sin (k(\varphi-\alpha)) d\varphi\\
&\quad+ \dfrac{k}{\pi}\int_{1}^{r}(\dfrac{1}{t^k}-\dfrac{t^k}{r^{2k}})\max_{K}\sum_{j \in K} \log \dfrac 
{\|f(te^{i\alpha})\|}{|(a_j,f)(te^{i\alpha})|} \dfrac{dt}{t} \\
&\quad+ \dfrac{k}{\pi}\int_{1}^{r}(\dfrac{1}{t^k}-\dfrac{t^k}{r^{2k}})\max_{K}\sum_{j \in K} \log \dfrac 
{\|f(te^{i\beta})\|}{|(a_j,f)(te^{i\beta})|} \dfrac{dt}{t}+O(1).
\end{align*}
\end{lemma}
\begin{proof} 
Let $a_{j} = (a^{j}_0,\dots,a^{j}_n)$  be the associated vector of $H_{j}$, $1\leqslant j \leqslant q,$ 
and let $\mathcal T$ be the set of all injective maps $\mu:\{ 0,1,\dots,n\} \longrightarrow \{1,\dots ,q\}$. 
By hypothesis, $ H_{1},\dots, H_{q}$ are in general position for any $\mu \in \mathcal T$, then the 
vectors $a_{\mu(0)},\dots,a_{\mu(n)}$ are linearly independent.  

Let $\mu \in T$, we have
\begin{align}\label{ct3.10}
(f,a_{\mu(t)})=a^{\mu(t)}_0f_{0}+\dots +a^{\mu(t)}_nf_{n},\  t=0,1,\dots,n.
\end{align}
Solve the system of linear equations (\ref{ct3.10}), we get
$$ f_{t}=b_{0}^{\mu(t)}(a_{0}^{\mu(t)},f)+\dots +b_{n}^{\mu(t)}(a_{n}^{\mu(t)},f), \ t=0,1,\dots,n,$$
where$ \biggl(b_{j}^{\mu(t)}\biggl)_{t,j=0}^{n}$ is the inverse matrix of $ \biggl(a_{j}^{\mu(t)}\biggl)_{t,j=0}^{n}.$ So there is a constant $C_\mu$ satisfying
$$\| f(z)\| \leqslant C_\mu \max_{0\leqslant t \leqslant n}| (a_{\mu(t)},f)(z)|.$$
Set $C = \max\limits_{\mu \in \mathcal T} C_\mu$. Then for any $\mu \in \mathcal T$, we have
$$\| f(z)\| \leqslant C \max_{0\leqslant t \leqslant n}| (a_{\mu(t)},f)(z)|.$$
For any $z\in\overline \Omega (\alpha, \beta)$, there exists the mapping $\mu\in \mathcal  T$ such that
$$ 0 <  |(a_{\mu(0)},f)(z)| \leqslant  |(a_{\mu(1)},f)(z) | \leqslant \dots .\leqslant | (a_{\mu(n)},f)(z)| \leqslant
 | (a_{j},f)(z)| ,$$ 
 for $j \notin \{\mu(0),\dots,\mu(n)\}$. Hence
$$\prod_{j=1}^{q}\dfrac {\| f(z)\|}{| (a_{j},f)(z)|} \leqslant  C ^{q-n-1}\max _{\mu\in \mathcal  T}\prod_{t=0}^{n}\dfrac {\| f(z)\|}{| (a_{\mu(t)},f)(z)|}.$$
We have

\begin{align*}
\sum_{j=1}^{q}[&A_{\alpha\beta, f}(r,H_{j})+B_{\alpha\beta, f}(r,H_{j})]\\
&=\sum_{j=1}^{q}\dfrac{2k}{\pi r^k}\int_{\alpha}^{\beta}(\dfrac{1}{t^k}-\dfrac{t^k}{r^{2k}}) \log \dfrac {\|f(re^{i\varphi})\|}{|( a_{j},f)
(re^{i\varphi})|}sin (k(\varphi-\alpha)) d\varphi\\
&\quad +\sum_{j=1}^{q}\dfrac{k}{\pi }\int_{1}^{r} (\dfrac{1}{t^k}-\dfrac{t^k}{r^{2k}})\log \dfrac {\|f(te^{i\alpha})\|}{|( a_{j},f)
(te^{i\alpha})|}\dfrac{dt}{t}\\
&\quad +\sum_{j=1}^{q}\dfrac{k}{\pi }\int_{1}^{r}(\dfrac{1}{t^k}-\dfrac{t^k}{r^{2k}}) 
\log \dfrac {\|f(te^{i\beta})\|}{|( a_{j},f)
(te^{i\beta})|}\dfrac{dt}{t}
\end{align*}
\begin{align*}
&= \dfrac{2k}{\pi r^k}\int_{\alpha}^{\beta}(\dfrac{1}{t^k}-\dfrac{t^k}{r^{2k}}) 
\log \prod_{j=1}^{q}\dfrac {\|f(re^{i\varphi})\|}
{|( a_{j},f)(re^{i\varphi})|}sin (k(\varphi-\alpha)) d\varphi\\
&\quad+\dfrac{k}{\pi }\int_{1}^{r} (\dfrac{1}{t^k}-\dfrac{t^k}{r^{2k}})\log \prod_{j=1}^{q}\dfrac {\|f(te^{i\alpha})\|}
{|( a_{j},f)(te^{i\alpha})|}\dfrac{dt}{t}\\
&\quad+\dfrac{k}{\pi }\int_{1}^{r} (\dfrac{1}{t^k}-\dfrac{t^k}{r^{2k}})\log \prod_{j=1}^{q}\dfrac {\|f(te^{i\beta})\|}
{|( a_{j},f)(te^{i\beta})|}\dfrac{dt}{t}.
\end{align*}
Thus
\begin{align*}
\sum_{j=1}^{q}[&A_{\alpha\beta, f}(r,H_{j})+B_{\alpha\beta, f}(r,H_{j})]\\
&\le \dfrac{2k}{\pi r^k}\int_{\alpha}^{\beta}(\dfrac{1}{t^k}-\dfrac{t^k}{r^{2k}}) \log \max _{\mu\in \mathcal  T}\prod_{t=0}^{n}\dfrac {\|f(re^{i\varphi})\|}
{|( a_{\mu(t)},f)(re^{i\varphi})|}sin (k(\varphi-\alpha)) d\varphi\\
&\quad+\dfrac{k}{\pi }\int_{1}^{r} (\dfrac{1}{t^k}-\dfrac{t^k}{r^{2k}})\log \max _{\mu\in \mathcal  T}\prod_{t=0}^{n}\dfrac {\|f(te^{i\alpha})\|}
{|( a_{\mu(t)},f)(te^{i\alpha})|}\dfrac{dt}{t}\\
&\quad+\dfrac{k}{\pi }\int_{1}^{r}(\dfrac{1}{t^k}-\dfrac{t^k}{r^{2k}}) \log \max _{\mu\in \mathcal  T}\prod_{t=0}^{n}
\dfrac {\|f(te^{i\beta})\|}{|( a_{\mu(t)},f)(te^{i\beta})|}\dfrac{dt}{t}+O(1).
\end{align*}
Therefore, we conclude that 
\begin{align*}
\sum_{j=1}^{q}[A_{\alpha\beta, f}(r,H_{j})&+B_{\alpha\beta, f}(r, H_j)] \\
&\le \dfrac{2k}{\pi r^{k}}
\int_{\alpha}^{\beta}(\dfrac{1}{t^k}-\dfrac{t^k}{r^{2k}})\max_{K}\sum_{j \in K} \log \dfrac {\|f(re^{i\varphi})\|}{|(a_j,f)
(re^{i\varphi})|}sin (k(\varphi-\alpha)) d\varphi\\
&\quad+ \dfrac{k}{\pi}\int_{1}^{r}(\dfrac{1}{t^k}-\dfrac{t^k}{r^{2k}})\max_{K}\sum_{j \in K} \log \dfrac 
{\|f(te^{i\alpha})\|}{|(a_j,f)(te^{i\alpha})|} \dfrac{dt}{t} \\
&\quad+ \dfrac{k}{\pi}\int_{1}^{r}(\dfrac{1}{t^k}-\dfrac{t^k}{r^{2k}})\max_{K}\sum_{j \in K} \log \dfrac 
{\|f(te^{i\beta})\|}{|(a_j,f)(te^{i\beta})|} \dfrac{dt}{t}+O(1).
\end{align*}
This is  conclusion of Lemma \ref{lm12}.
\end{proof}

By argument as Lemma \ref{lm11} and Lemma \ref{lm12}, we are easy to get results as following:

\begin{lemma}\label{lm15} Let $ f=(f_{0}:\dots:f_{n}) : \Omega(\alpha, \beta) \ax \PP^{n}(\C)$  be a linearly non-degenerate
 holomorphic curve  and $H_{1},\dots , H_{q}$ be hyperplanes in $\PP^{n}(\C)$ in general position. Then we have
\begin{align*}
\| \quad \dfrac{1}{2\pi}\int_{arcsin r^{-k}}^{\pi-arcsin r^{-k}}\max_{K}\sum_{j \in K}& \log \dfrac
 {\|f(rsin^{k^{-1}}
\varphi e^{i(\alpha+k^{-1}\varphi)})\|}{|(a_j,f)
(rsin^{k^{-1}}
\varphi e^{i(\alpha+k^{-1}\varphi)})|}\dfrac{d\varphi}{r^ksin^2\varphi}\\
&\leqslant (n+1)\mathfrak T_{\alpha\beta, f}(r)- \mathfrak N_{\alpha\beta, W}(r,0)\\
 &\quad +O(\log \mathfrak T_{\alpha\beta, f}(r)+\log r).
\end{align*}
Here the maximum is taken over all subsets $K$ of $\{1,\dots,q\}$ such that $a_{j}$, $j\in K$,
 are linearly independent.
\end{lemma}

\begin{lemma}\label{lm16} Let $ f=(f_{0}:\dots:f_{n}) : \Omega(\alpha, \beta) \longrightarrow \PP^{n}(\C)$  be a linearly
 non-degenerate holomorphic curve  and $H_{1},\dots ,H_{q}$ be hyperplanes in $\PP^{n}(\C)$ in general position. 
Let $a_j$ be the vector associated with $H_j$ for $j=1,\dots,q$. Then
\begin{align*}
\sum_{j=1}^{q}\mathfrak m_{\alpha\beta, f}(r, H_j)&\le \dfrac{1}{2\pi}\int_{arcsin r^{-k}}^{\pi-arcsin r^{-k}}\max_{K}\sum_{j \in K}
 \log \dfrac{\|f(rsin^{k^{-1}}
\varphi e^{i(\alpha+k^{-1}\varphi)})\|}{|(a_j,f)(rsin^{k^{-1}}
\varphi e^{i(\alpha+k^{-1}\varphi)})|}\dfrac{d\varphi}{r^ksin^2\varphi}\\
&\quad+O(1).
\end{align*}
\end{lemma}

\begin{proof}[Proof of Theorem \ref{th4} and Theorem \ref{th6}]
First, we prove the Theorem \ref{th4}. Using Lemma \ref{lm11} and Lemma \ref{lm12}, we have
\begin{align}\label{ct3.7}
\| \sum_{j=1}^{q}[&A_{\alpha\beta, f}(r,H_{j})+B_{\alpha\beta, f}(r,H_{j})]\notag\\
&\le \dfrac{2k}{\pi r^k}\int_{\alpha}^{\beta}(\dfrac{1}{t^k}-\dfrac{t^k}{r^{2k}}) \log \max _{\mu\in \mathcal  T}\prod_{t=0}^{n}\dfrac {\|f(re^{i\varphi})\|}
{|( a_{\mu(t)},f)(re^{i\varphi})|}sin (k(\varphi-\alpha)) d\varphi\notag\\
&\quad+\dfrac{k}{\pi }\int_{1}^{r} \log \max _{\mu\in \mathcal  T}\prod_{t=0}^{n}\dfrac {\|f(te^{i\alpha})\|}
{|( a_{\mu(t)},f)(te^{i\alpha})|}\dfrac{dt}{t}\notag\\
&\quad+\dfrac{k}{\pi }\int_{1}^{r} \log \max _{\mu\in \mathcal  T}\prod_{t=0}^{n}
\dfrac {\|f(te^{i\beta})\|}{|( a_{\mu(t)},f)(te^{i\beta})|}\dfrac{dt}{t}+O(1)\notag\\
&\le (n+1)S_{\alpha\beta, f}(r)-C_{\alpha\beta, W}(r,0)+O(\log T_f(r)+\log r).
\end{align}
By Corollary \ref{th1}, we get that
$$S_{\alpha\beta,f}(r)=A_{\alpha\beta, f}(r,H_{j})+B_{\alpha\beta, f}(r,H_{j})+C_{\alpha\beta,f}(r, H_j)+O(1)$$
for any $j \in\{1,\dots,q\}$. So from (\ref{ct3.7}), we have
\begin{align}\label{ct3.8}
\| \quad (q-n-1)S_{\alpha\beta,f}(r)\leqslant \sum_{j=1}^{q}C_{\alpha\beta,f}(r,H_{j})- C_{\alpha\beta, W}(r,0)
+O(\log r + \log T_{f}(r)).
\end{align}
For  $z_0 \in \overline \Omega(\alpha, \beta)$, we may assume that $(a_{j},f)$ vanishes at $z_0$ for $ 1\leqslant j \leqslant q_{1}$, $(a_{j},f)$ does not vanish at $z_0$ for $ j > q_{1}$. Hence, there exists a integer $k_{j}$ and nowhere vanishing holomorphic function $ g_{j}$ in neighborhood $U$ of $z$ such that
$$(a_{j},f)(z)=(z-z_0)^{k_{j}}g_{j}(z), \text{ for } j=1,\dots ,q,$$
here $k_{j}=0$ for $q_{1} <j \leqslant q$. We may assume that $ k_{j}\geqslant n$ for  $1\leqslant j \leqslant q_{0}$, and $1\leqslant k_{j}< n$ for $q_{0}< j \leqslant q_{1}.$ By property of the Wronskian, we have
$$W(f)=C.W((a_{\mu(0)},f),\dots .,(a_{\mu(n)},f))=\prod_{j=1}^{q_{0}}(z-z_0)^{k_{j}-n}h(z),$$
where $h(z)$ is holomorphic function on $U$. Then $W(f)$ is vanishes at $z_0$ with order at least $$\sum\limits_{j=1}^{q_{0}} (k_{j}-n)=\sum\limits_{j=1}^{q_{0}} k_{j}-q_{0}n.$$
By the definition of $ C_{\alpha\beta, f}(r,H),\  C_{\alpha\beta, W}(r,0)$ and $ C_{\alpha\beta, f}^{n}(r,H)$, we have
\begin{align*}
\sum_{j=1}^{q}&C_{\alpha\beta, f}(r,H_{j})-C_{\alpha\beta, W}(r,0)\leqslant \sum_{j=1}^{q}C^{n}_{\alpha\beta,f}
(r,H_{j}) +O(1).
\end{align*}
So from (\ref{ct3.8}), we have
$$\| \quad (q-n-1)S_{\alpha\beta, f}(r)\leqslant \sum_{j=1}^{q}C^{n}_{\alpha\beta, f}(r,H_{j})+O(\log r+\log 
T_{f}(r)).$$
The proof of Theorem \ref{th4} is completed. 

Next, we prove the Theorem \ref{th6}. By Lemma \ref{lm15} and Lemma \ref{lm16}, we have
\begin{align}\label{ct3.7a}
\| \sum_{j=1}^{q}\mathfrak m_{\alpha\beta, f}(r,H_{j})&\le \quad \dfrac{1}{2\pi}\int_{arcsin r^{-k}}^{\pi-arcsin r^{-k}}
\max_{K}\sum_{j \in K} \log \dfrac
 {\|f(rsin^{k^{-1}}
\varphi e^{i(\alpha+k^{-1}\varphi)})\|}{|(a_j,f)
(rsin^{k^{-1}}
\varphi e^{i(\alpha+k^{-1}\varphi)})|}\dfrac{d\varphi}{r^ksin^2\varphi}\\
&\leqslant (n+1)\mathfrak T_{\alpha\beta, f}(r)- \mathfrak N_{\alpha\beta, W}(r,0)+O(\log \mathfrak T_{\alpha\beta, f}(r)+\log r).
\end{align}
 Corollary 2 gives that
$$\mathfrak T_{\alpha\beta, f}(r)=\mathfrak m_{\alpha\beta, f}(r,H_{j})+\mathfrak N_{\alpha\beta, f}(r,H_{j})+O(1)$$
for any $j \in\{1,\dots,q\}$. Hence from (\ref{ct3.7a}), we obtain
\begin{align}\label{ct3.8a}
\| \quad (q-n-1)\mathfrak T_{\alpha\beta, f}(r)\leqslant \sum_{j=1}^{q}\mathfrak N_{\alpha\beta, f}(r,H_{j})-
 \mathfrak N_{\alpha\beta, W}(r,0)
+O(\log r + \log \mathfrak T_{\alpha\beta, f}(r)).
\end{align}

For  $z_0 \in \Omega(\alpha, \beta; r)=\Omega(\alpha, \beta)\cap\{1<|z|<r\},$ we may suppose that $(a_{j},f)$ vanishes at $z_0$ for $ 1\leqslant j \leqslant q_{1}$, $(a_{j},f)$ does not vanish at $z_0$ for $ j > q_{1}$. Hence, there exists a integer $k_{j}$ and nowhere vanishing holomorphic function $ g_{j}$ in neighborhood $U$ of $z$ such that
$$(a_{j},f)(z)=(z-z_0)^{k_{j}}g_{j}(z), \text{ for } j=1,\dots ,q,$$
here $k_{j}=0$ for $q_{1} <j \leqslant q$. We may assume that $ k_{j}\geqslant n$ for  $1\leqslant j \leqslant q_{0}$, and $1\leqslant k_{j}< n$ for $q_{0}< j \leqslant q_{1}.$ By property of the Wronskian, we have
$$W(f)=C.W((a_{\mu(0)},f),\dots .,(a_{\mu(n)},f))=\prod_{j=1}^{q_{0}}(z-z_0)^{k_{j}-n}h(z),$$
where $h(z)$ is holomorphic function on $U$. Then $W(f)$ is vanishes at $z_0$ with order at least $$\sum\limits_{j=1}^{q_{0}} (k_{j}-n)=\sum\limits_{j=1}^{q_{0}} k_{j}-q_{0}n.$$
By the definition of $ \mathfrak N_{f}(r,H),\mathfrak N_{W}(r,0)$ and $ \mathfrak N_{f}^{n}(r,H)$, we have
\begin{align*}
\sum_{j=1}^{q}&\mathfrak N_{\alpha\beta, f}(r,H_{j})-\mathfrak N_{\alpha\beta, W}(r,0)\leqslant 
\sum_{j=1}^{q}\mathfrak N^{n}_{\alpha\beta, f}(r,H_{j}) +O(1).
\end{align*}
Thus from (\ref{ct3.8a}), we get the inequality
$$\| \quad (q-n-1)\mathfrak T_{\alpha\beta, f}(r)\leqslant \sum_{j=1}^{q}\mathfrak N^{n}_{\alpha\beta, f}
(r,H_{j})+O(\log r+\log 
\mathfrak T_{\alpha\beta, f}(r)).$$
This is statement of Theorem \ref{th6}. 
\end{proof}

\begin{proof}[Proof of Theorem \ref{Th10}]

Let ${\large \text{f}}= (f_0:\dots: f_N)$ be a reduced representation of $f,$ where $f_0, \dots, f_N$ are entire functions on 
$\Omega(\alpha,\beta)$ and have no common zeros. We consider the function $\phi_i=Q_i\circ {\large \text{f}}=Q_i(f_0, \dots, f_N), 0\le i\le N.$ Let 
 $F=(\phi_0f_0^n:\dots: \phi_Nf_N^n).$ Since the hypersurfaces  $\{\mathcal H_i^nQ_i=0\}, 0\le i\le N,$ are located in general position 
in $\mathbb P^{N}(\mathbb C),$ then $F: \Omega(\alpha, \beta)\to \mathbb P^N(\mathbb C)$ is a holomorphic curve. Let 
 $\mathfrak H_i, 0\le i\le N,$ be the hypersurface defined by $\{\mathcal H_i^nQ_i=0\}, 0\le i\le N.$
From the hypothesis  $\mathfrak H_0, \dots, \mathfrak H_N$ are in general position, i.e.
$$ \text{supp}\mathfrak H_{0} \cap \dots \cap \text{supp}\mathfrak H_N= \emptyset.$$
Thus by Hilbert's Nullstellensatz \cite{VW}, for
any integer k, $0 \le k \le N,$ there is an integer $m_k > n+d$ such that
$$ x_k^{m_k}=\sum_{i=0}^{N}b_{i}(x_0, \dots, x_N)  \mathcal H_i^n(x_0, \dots, x_N)Q_{i}(x_0, \dots, x_N),$$
where $b_{0}, \dots, b_N$ are homogeneous forms with coefficients in $\mathbb C$ of degree $m_k-(n+d).$  
This implies
$$ 
|f_k(z )|^{m_k}\le c_1||{\large \text{f}}(z )||^{m_k-(n+d)}\max\{|\mathcal H_0^nQ_{0}({\large \text{f}}(z ))|, \dots, |\mathcal H_N^nQ_{N}
({\large \text{f}}(z ))|\},$$
where $c_1$ is a positive constant depending only on the coefficients of $b_{i}, 0\le i\le N, 0\le k\le N,$ thus depending only on the coefficients of $Q_i, 0\le i\le N.$ Therefore,
\begin{align}\label{cta2}
||{\large \text{f}}(z )||^{n+d}\le c_1\max\{|\mathcal H_0^nQ_{0}({\large \text{f}}(z ))|, \dots, |\mathcal H_N^nQ_{N}({\large \text{f}}(z ))|\}.
\end{align}

From (\ref{cta2}) and the First Main Theorem, we have
\begin{align}\label{41}
\mathfrak T_{\alpha\beta, F}(r)&\ge (n+d)\mathfrak T_{\alpha\beta, f}(r)+O(1)\notag\\
&\ge (n+d-(N+1)d)\mathfrak T_{\alpha\beta, f}(r)+\sum_{i=0}^{N}\mathfrak N_{\alpha\beta, f}(r, D_i)+O(1)\notag\\
&= (n-Nd)\mathfrak T_{\alpha\beta, f}(r)+\sum_{i=0}^{N}\mathfrak N_{\alpha\beta, f}(r, D_i)+O(1).
\end{align}
On the other hand, by applying Theorem \ref{th6} to $F$, and the hyperplanes 
$$ H_i=\{y_i=0\}, 0\le i\le N,$$
and
$$ H_{N+1}=\{y_0+\dots+y_{N}=0\} $$
yields
\begin{align}\label{42}
\|\mathfrak T_{\alpha\beta, F}(r)\le \sum_{i=0}^{N+1}\mathfrak N_{\alpha\beta, F}^{N}(r, H_i)+o(\mathfrak T_{\alpha\beta, f}(r)).
\end{align}
We have
\begin{align*}
\mathfrak N_{\alpha\beta, F}^{N}(r, H_i)\le \mathfrak N_{\alpha\beta, f}^{N}(r, D_i)+\mathfrak N_{\alpha\beta}^{N}(r, \dfrac{1}{f_i^n})
\end{align*}
for all $i=0, \dots, N,$ where $\mathfrak N^N(r, \dfrac{1}{g})$ is counting function with level of truncation $N$ of $g.$ Hence
\begin{align}\label{43}
\mathfrak N_{\alpha\beta, F}^{N}(r, H_i)&\le \mathfrak N_{\alpha\beta, f}^{N}(r, D_i)+N\overline {\mathfrak N}_{\alpha\beta}
(r, \dfrac{1}{f_i^n})\notag\\
&\le \mathfrak N_{\alpha\beta, f}^{N}(r, D_i)+N\mathfrak T_{\alpha\beta, f}(r)+O(1)
\end{align}

for all $i=0, \dots, N.$ Also note $\mathfrak N_{\alpha\beta, F}^{N}(r, H_{N+1})=\mathfrak N_{\alpha\beta, f}^{N}(r, D).$ 
By combining (\ref{41}) to  (\ref{43}), we obtain 
\begin{align*}
\|(n-(d+N+1)N)\mathfrak T_{\alpha\beta, f}(r)&+\sum_{i=0}^{N}(\mathfrak N_{\alpha\beta, f}(r, D_i)-\mathfrak 
N_{\alpha\beta, f}^{N}(r, D_i))\\
&\le \mathfrak 
N_{\alpha\beta, f}^{N}(r, D)+o(\mathfrak T_{\alpha\beta, f}(r)).
\end{align*}
\end{proof}

\begin{proof}[Proof of Theorem \ref{Th11}]
We suppose that $f\not\equiv g,$ then there are two numbers $i, j \in  \{0, \dots , N\},$ $ i\ne j$ such that
$f_{i}g_{j}\not\equiv f_{j}g_{i}.$ Assume that $z_0 \in \Omega(\alpha, \beta)$ is a zero of $Q(f),$ where $Q$ is a homogeneous
 defining $D.$ From condition $f(z)=g(z)$ when $z\in f^{-1}(D) \cup g^{-1}(D),$ we get $f(z_0)=g(z_0).$ This implies $z_0$ is a zero of $\dfrac{f_{i}}{f_{j}}-\dfrac{g_{i}}{g_{j}}.$ Therefore, we have
\begin{align*}
 \mathfrak {N}_{\alpha\beta, f}^N(r, D) \le N\mathfrak {N}_{\alpha\beta, f}^{1}(r, D)&\le N \mathfrak {N}_{\alpha\beta}(r, \dfrac{1}{ \dfrac{f_{i}}{f_{j}}-\dfrac{g_{i}}{g_{j}}})\\
&\le N(\mathfrak {T}_{\alpha\beta, f}(r)+\mathfrak {T}_{\alpha\beta, g}(r))+O(1).
\end{align*}
Apply to Theorem \ref{Th10}, we obtain
\begin{align}\label{51}
  \|(n-(d+N+1)N)\mathfrak {T}_{\alpha\beta, f}(r)\le N(\mathfrak {T}_{\alpha\beta, f}(r)+\mathfrak {T}_{\alpha\beta, g}(r))+o(\mathfrak {T}_{\alpha\beta, f}(r)).
\end{align}
Similarly, we have
\begin{align}\label{52}
  \|(n-(d+N+1)N)\mathfrak {T}_{\alpha\beta, g}(r)\le N(\mathfrak {T}_{\alpha\beta, f}(r)+\mathfrak {T}_{\alpha\beta, g}(r))+o(\mathfrak {T}_{\alpha\beta, g}(r)).
\end{align}
Combining (\ref{51}) and (\ref{52}), we get
$$ \|(n-(d+N+3)N)(\mathfrak {T}_{\alpha\beta, f}(r)+\mathfrak {T}_{\alpha\beta, g}(r)) \le o(\mathfrak {T}_{\alpha\beta, f}(r))+o(\mathfrak {T}_{\alpha\beta, g}(r)).$$
This is a contradiction with $n>(d+N+3)N.$ Hence $f\equiv g.$
\end{proof}

\noindent{\bf Acknowledgements}\\
 The author thanks to the Proffesor Jian-Hua Zheng for very helpful
 comments and useful suggestions in this paper.

\end{document}